\documentclass[a4paper,11pt]{amsart}
\usepackage[hmargin=36mm,tmargin=42mm,bmargin=38mm]{geometry}
\usepackage{amsfonts,amsmath,amssymb,amsthm,enumerate,mathtools}
\usepackage[hidelinks]{hyperref}

\usepackage{tikz}\usetikzlibrary{arrows}
\usepackage[utf8]{inputenc}

\newtheorem{thm}{Theorem}[section]
\newtheorem{prop}[thm]{Proposition}
\newtheorem{cor}[thm]{Corollary}
\newtheorem{lem}[thm]{Lemma}

\theoremstyle{definition}
\newtheorem{defn}[thm]{Definition}
\newtheorem{exas}[thm]{Examples}
\newtheorem{exa}[thm]{Example}
\newtheorem{que}[thm]{Question}

\theoremstyle{remark}
\newtheorem{rem}[thm]{Remark}

\newcommand{\on}{\operatorname}

\newcommand{\F}{\mathbb{F}}\newcommand{\R}{\mathbb{R}}
\newcommand{\N}{\mathbb{N}}\newcommand{\Z}{\mathbb{Z}}
\newcommand{\cl}{\on{cl}}
\newcommand{\PG}{\on{PG}}\newcommand{\AG}{\on{AG}}

\newcommand{\bG}{\mathbf G}\newcommand{\bU}{\mathbf U}
\newcommand{\bP}{\mathbf P}\newcommand{\bA}{\mathbf A}

\newcommand{\cP}{\mathcal P}\newcommand{\cL}{\mathcal L}
\newcommand{\cU}{\mathcal U}
\newcommand{\cX}{\mathcal X}

\newcommand{\cB}{\mathcal B}\newcommand{\cC}{\mathcal C}
\newcommand{\cD}{\mathcal D}\newcommand{\cG}{\mathcal G}
\newcommand{\cH}{\mathcal H}\newcommand{\cI}{\mathcal I}

\begin{document}

\title{Designs and codes in affine geometry}

\author{Jens Zumbrägel}

\address{School of Computer and Communication Sciences, EPFL, Station 14,
  1015 Lausanne, Switzerland; email: jens.zumbragel@epfl.ch}

\date{\today}

\begin{abstract}
  Classical designs and their (projective) $q$-analogs can both be
  viewed as designs in matroids, using the matroid of all subsets of a
  set and the matroid of linearly independent subsets of a vector
  space, respectively.  Another natural matroid is given by the point
  sets in general position of an affine space, leading to the concept
  of an affine design.  Accordingly, a $t$-$(n, k, \lambda)$ affine
  design of order~$q$ is a collection~$\cB$ of $(k-1)$-dimensional
  spaces in the affine geometry $\textbf{A} = \AG(n-1, q)$ such that
  each $(t-1)$-dimensional space in~$\textbf{A}$ is contained in
  exactly~$\lambda$ spaces of~$\cB$.  In the case $\lambda = 1$, as
  usual, one also refers to an affine Steiner system $S(t, k, n)$.

  In this work we examine the relationship between the affine and the
  projective $q$-analogs of designs.  The existence of affine Steiner
  systems with various parameters is shown, including the affine
  $q$-analog $S(2, 3, 7)$ of the Fano plane.  Moreover, we consider
  various distances in matroids and geometries, and we discuss the
  application of codes in affine geometry for error-control in a
  random network coding scenario.
\end{abstract}

\maketitle


\section{Introduction}

In recent years the $q$-analogs of combinatorial designs and Steiner
systems enjoyed a considerable interest due to their relation with
error-control in random network coding~\cite{KK08}.  In this context a
remarkable result has been the discovery of a nontrivial $2$-ary $S(2,
3, 13)$ Steiner system~\cite{BEOVW13}, whereas, perhaps surprisingly,
no other $q$-analog of a Steiner system has been found since.

The present work aims to offer a new perspective on $q$-analog
structures.  While the $q$-analog designs studied so far can be seen
as living in projective geometry, we propose to consider the setup of
affine geometry.  A suitable framework for both the ``projective'' and
the ``affine'' analogs of designs is provided by the theory of
matroids -- or more precisely, of so-called perfect matroid designs in
which all flats of same rank have equal cardinality.  As it turns out
there exist families of affine $q$-analog Steiner systems for several
parameters.  Besides being quite natural from a pure mathematical
point of view, this approach is also viable for random network coding
applications, when propagating in a network random affine combinations
instead of linear ones, cf.~\cite{GG11}.

Here is an outline of the paper.  In Section~\ref{sec:matroids} and
Section~\ref{sec:geometry} we provide an account of matroids and of
finite geometry, as suitable for introducing the concept at hand.  The
discussion of projective and affine designs as well as their
relationship then follows in Section~\ref{sec:designs}.  Finally,
metric aspects and possible applications to the random network coding
scenario are outlined in Section~\ref{sec:coding}.


\section{Matroids}\label{sec:matroids}

The notion of a matroid was introduced by Whitney~\cite{Wh35} as an
abstraction of the concept of linear independence.  It serves as a
suitable framework for studying $q$-analogs of designs, as well as for
developing metric and coding aspects in a general setting~\cite{GG11}.
In this section we provide a brief treatment of some relevant aspects
of matroid theory.  For background reading and further details on
matroids we refer to Oxley's monograph~\cite{Ox11}, while a concise
introduction to designs in matroids can be found in~\cite{CD06}.

\begin{defn}
  A \emph{matroid} is a pair $(S, \cI)$, where~$S$ is a finite set
  and~$\cI$ is a nonempty family of subsets of~$S$ satisfying
  \begin{enumerate}[(i)]
  \item if $I \in \cI$ and $J \subseteq I$, then $J \in \cI$;
  \item if $I, J \in \cI$ and $|I| < |J|$, there is $x \in J
    \setminus I$ with $I \cup \{ x \} \in \cI$ (\emph{exchange
      axiom}).
  \end{enumerate}
  A subset~$I$ of~$S$ is called \emph{independent} if $I \in \cI$,
  otherwise \emph{dependent}.
\end{defn}

\begin{exas}
  The following are examples for matroids.
  \begin{enumerate}
  \item The \emph{free matroid} is the matroid $(S, \cP(S))$, where~$S$
    is a finite set and $\cP(S)$ denotes the power set of~$S$.
  \item For a finite vector space~$V$, the \emph{vector matroid} is the
    matroid~$(V, \cI)$, where $\cI$ is the family of all
    linearly independent subsets of~$V$.
  \item Let $G = (V, E)$ be an undirected graph with finite vertex
    set~$V$ and edge set $E \subseteq {V \choose 2}$.  Then $(E, \cI)$
    is a \emph{graphic matroid}, where a subset of the edges is
    independent if and only if it contains no cycle.
  \end{enumerate}
\end{exas}

Given a matroid $M = (S, \cI)$ and any subset~$X$ of~$S$ the {\em
  restriction} $M|X \coloneqq (X, \cI \cap \cP(X))$ is again a
matroid.  Finite restrictions of vector spaces or, equivalently, the
column sets of matrices were among the original motivating examples of
matroids, besides the graphic matroids.

In a matroid $M = (S, \cI)$, a maximal independent set is called a
\emph{basis}, and a minimal dependent set is a \emph{circuit}.  It
follows from the exchange axiom that all bases have the same
cardinality, and this number is called the \emph{rank} of the
matroid~$M$.  Moreover, the \emph{rank} $\rho(X)$ of a subset~$X$
of~$S$ is the rank of the restricted matroid $M|X$, i.e., the
cardinality of a maximal independent subset of~$X$.  

\begin{prop}[{cf.~\cite[Cor.~1.3.4]{Ox11}}]
  The rank function $\rho \colon \cP(S) \to \Z$ of a matroid $(S, \cI)$
  satisfies the following properties for all subsets $X, Y$ of~$S$:
  \begin{enumerate}[(i)]
  \item $0 \le \rho(X) \le | X |$,
  \item $X \subseteq Y$ implies $\rho(X) \le \rho(Y)$,
  \item $\rho(X \cup Y) + \rho(X \cap Y) \le \rho(X) + \rho(Y)$
    (submodular inequality).
  \end{enumerate}
  Conversely, any mapping $\rho \colon \cP(S) \to \Z$ with these properties
  is the rank function of a matroid $(S, \cI)$, where
  $\cI \coloneqq \{ I \subseteq S \mid \rho (I) = | I | \}$.
\end{prop}

The mapping $\cl \colon \cP(S) \to \cP(S)$ defined by
\[ \cl(X) \coloneqq \{ x \in S \mid \rho(X \cup \{ x \}) = \rho(X) \} \]
for $X \in \cP(S)$ turns out to be a \emph{closure operator}, i.e., it
holds $X \subseteq \cl(X) = \cl(\cl(X))$ and $\cl(X) \subseteq \cl(Y)$
if $X \subseteq Y$, for all $X, Y \in \cP(S)$.  A subset~$E$ of~$S$
satisfying $E = \cl(E)$ is called a \emph{flat}, or a $k$-\emph{flat}
if its rank~$\rho(E)$ is~$k$.

\begin{prop}[{cf.~\cite[Cor.~1.4.6]{Ox11}}]\label{prop:closure}
  The closure operator $\cl \colon \cP(S) \to \cP(S)$ of a matroid
  $(S, \cI)$ obeys the following for any flat~$E$
  and any $x, y \in S \setminus E$:
  \[ y \in \cl(E \cup \{ x \}) \ \text{ implies } \ x \in
  \cl(E \cup \{ y \}) \quad (\text{exchange property}) . \] 
  Conversely, any closure operator $\cl \colon \cP(S) \to \cP(S)$ with
  the exchange property is the closure operator of a matroid $(S,
  \cI)$, where $\cI \coloneqq \{ I \subseteq S \mid \forall x \in I
  :\, x \notin \cl(X \setminus \{ x \} ) \}$.
\end{prop}

A subset~$X$ of~$S$ is called \emph{generating} if $\cl(X) = S$.  It
is easy to verify that the family of bases in a matroid coincides with
the independent generating sets and also with the minimal generating
sets.


\subsection*{The lattice of flats}

For a matroid $M = (S, \cI)$ let $L(M) \coloneqq \{ E \subseteq S \mid
\cl(E) = E \}$ denote its collection of flats.  When ordered by
inclusion this is a \emph{lattice}, i.e., for any flats $E, F$ there
exists a greatest lower bound or \emph{meet} $E \wedge F = E \cap F$
and a least upper bound or \emph{join} $E \vee F = \cl(E \cup F)$; see
Figure~\ref{fig:flats} for an example.

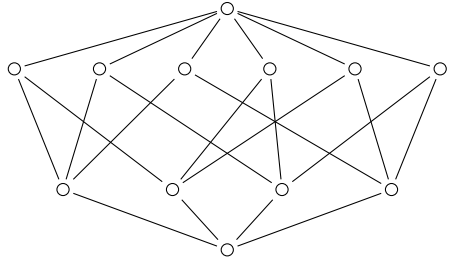
\begin{figure}
  \begin{tikzpicture}[scale=0.8]
    \node (o) at (0,-1) {}; \node (i) at (0,3) {};
    \node (a) at (-2.7,0) {}; \node (b) at (-0.9,0) {};
    \node (c) at  (0.9,0) {}; \node (d) at  (2.7,0) {};
    \node (ab) at (-3.5,2) {}; \node (ac) at (-2.1,2) {};
    \node (ad) at (-0.7,2) {}; \node (bc) at  (0.7,2) {};
    \node (bd) at  (2.1,2) {}; \node (cd) at  (3.5,2) {};
    \draw (o) circle (0.1); \draw (i) circle (0.1);
    \draw (a) circle (0.1); \draw (b) circle (0.1);
    \draw (c) circle (0.1); \draw (d) circle (0.1);
    \draw (ab) circle (0.1); \draw (ac) circle (0.1);
    \draw (ad) circle (0.1); \draw (bc) circle (0.1);
    \draw (bd) circle (0.1); \draw (cd) circle (0.1);
    \draw (o) edge (a); \draw (o) edge (b);
    \draw (o) edge (c); \draw (o) edge (d);
    \draw (a) edge (ab); \draw (b) edge (ab); \draw (a) edge (ac);
    \draw (c) edge (ac); \draw (a) edge (ad); \draw (d) edge (ad);
    \draw (b) edge (bc); \draw (c) edge (bc); \draw (b) edge (bd);
    \draw (d) edge (bd); \draw (c) edge (cd); \draw (d) edge (cd);
    \draw (ab) edge (i); \draw (ac) edge (i); \draw (ad) edge (i);
    \draw (bc) edge (i); \draw (bd) edge (i); \draw (cd) edge (i);
  \end{tikzpicture}
  \caption{The lattice of the flats in $\AG(2, 2)$.}\label{fig:flats}
\end{figure}

Notice that the rank function $\rho \colon L(M) \to \Z$ on the set of
flats is strictly isotone, i.e., $\rho(E) < \rho(F)$ if $E \subsetneq
F$, as well as submodular, i.e., $\rho(E \vee F) + \rho(E \wedge F)
\le \rho(E) + \rho(F)$ for flats $E, F$.  In this context the
following result is of interest, which has been shown in a more
general setup (see, e.g.,~\cite{HG72,Mo77}).  We include a direct
proof for convenience (see also \cite[Prop.~7, Cor.~2]{GG11}).

\begin{prop}\label{prop:metric}
  In any lattice~$L$ with a strictly isotone and submodular function
  $r \colon L \to \R$ a metric is given by $d(x, y) \coloneqq 
  2 r(x \vee y) - r(x) - r(y)$, for $x, y \in L$.
\end{prop}

\begin{proof}
  The properties of a metric are clear, except for the triangle
  inequality.  Let elements $x, t, y \in L$ be given.  Notice that
  \[ r(x \vee y) - r(x) \le r(x \vee y \vee t) - r(x) = r(x \vee y
  \vee t) - r(x \vee t) + r(x \vee t) - r(x), \] and the submodularity
  implies that \[ r(x \vee t \vee y) - r(x \vee t) \le r(t \vee y) -
  r((x \vee t) \wedge (t \vee y)) \le r(t \vee y) - r(t). \] Hence
  $r(x \vee y) - r(x) \le r(x \vee t) - r(x) + r(t \vee y) - r(t)$,
  and, by exchanging~$x$ and~$y$, we get $r(x \vee y) - r(y) \le r(x
  \vee t) - r(t) + r(t \vee y) - r(y)$.  Adding these two inequalities
  we obtain $d(x, y) \le d(x, t) + d(t, y)$, as desired.
\end{proof}

\begin{rem}
  The same proof reveals that another metric on~$L$ is given by $d'(x,
  y) \coloneqq r(x \vee y) - \min( r(x), r(y) )$, for $x, y \in L$.
\end{rem}

Accordingly, we associate for any matroid~$M$ on its set of
flats~$L(M)$ the metric $d_M(E, F) \coloneqq 2 \rho(E \vee F) -
\rho(E) - \rho(F)$, where $E, F$ are flats.

We mention also the following characterization.  Recall that in an
ordered set an element~$y$ \emph{covers} an element~$x$ if $x < y$ but
there is no element~$t$ such that $x < t < y$, and an \emph{atom} is
an element covering a least element.  The lattice of flats is
\emph{atomistic}, i.e., every element is a join of atoms, as well as
\emph{semimodular}, i.e., if~$E$ and~$F$ both cover $E \wedge F$, then
$E \vee F$ covers both~$E$ and~$F$; consequently, the Jordan-Dedekind
chain condition is satisfied, in fact, every maximal chain between
flats $E < F$ has length $\rho(F) - \rho(E)$.  Conversely, every
finite, atomistic and semimodular lattice can shown to be the lattice
of flats in a matroid (cf.~\cite[Thm.~1.7.5]{Ox11}).


\subsection*{Designs in matroids}

In order to define designs in matroids it is reasonable to restrict
the class of matroids.  A \emph{perfect matroid design} (PMD) is a
matroid~$M$ of some rank~$r$ for which any $i$-flat has the same
cardinality~$f_i$, where $0 \le i \le r$.  In this case, we say
that~$M$ is of \emph{type} $(f_0, \dots, f_r)$.

\begin{exas}
  The following are examples for perfect matroid designs.
  \begin{enumerate}
  \item The free matroid of rank~$n$, i.e., the matroid $(S, \cP(S))$,
    where~$S$ is a set of cardinality~$n$, is a PMD with $f_i = i$,
    for $0 \le i \le n$.
  \item For an $n$-dimensional vector space~$V$ over a finite
    field~$\F_q$, the vector matroid $(V, \cI)$ has rank~$n$ and its
    $i$-flats are the $i$-dimensional subspaces; it is a PMD with $f_i
    = q^i$, for $0 \le i \le n$.
  \end{enumerate}
\end{exas}

A \emph{loop} in a matroid $(S, \cI)$ is an element $x \in S$ such
that $\{ x \} \notin \cI$.  Two nonloops $x, y \in S$ are
\emph{parallel} if $\{ x, y \} \notin \cI$.  A matroid with no loops
and no distinct parallel elements is called \emph{geometric}.  From
any matroid one obtains a geometric matroid by deleting loops and
identifying parallel elements.  If~$M$ is a PMD of type $(f_0, \dots,
f_r)$, then its geometrization is a PMD of type $(f'_0, \dots, f'_r)$,
where $f'_i \coloneqq \smash{\frac{f_i - f_0} {f_1 - f_0}}$; hence,
$f'_0 = 0$ and $f'_1 = 1$.  In particular, the geometrization of a
vector space over~$\F_q$ is the corresponding projective space, in
which $f_i = [i]_q \coloneqq \smash{\frac{q^i - 1} {q - 1}}$.

\begin{defn}\label{defn:design-matroid}
  Let~$M$ be a PMD of rank~$n$.  A $t$-$(n, k, \lambda)$ \emph{design}
  in~$M$ is a collection~$\cB$ of $k$-flats in~$M$ such that each
  $t$-flat in~$M$ is contained in exactly~$\lambda$ members of~$\cB$.
\end{defn}

For the free matroid $M = (S, \cP(S))$ one obtains the classical
designs.

\begin{lem}
  Let~$M$ be a PMD of rank~$n$ and type $(f_0, \dots, f_n)$.  Any
  $t$-$(n, k, \lambda)$ design in~$M$ is also an $s$-$(n, k,
  \lambda_s)$ design for $s < t$, where \[ \lambda_s
  = \lambda \prod_{i=s}^{t-1} \frac{f_n - f_i} {f_k - f_i} . \]
\end{lem}

\begin{proof}
  By induction it suffices to show the statement for $s = t - 1$,
  which follows by adapting the double-counting argument from
  the classical case.  Let~$\cB$ be a $t$-$(n, k, \lambda)$ design,
  let~$E$ be any $s$-flat and consider all blocks $B \in \cB$ with 
  $E \subseteq B$, the number being $\lambda_s$.  Denoting, for $B \in
  \cB$ and $x \in S$, $\chi(x, B) = 1$ if $x \in B$ and $\chi(x, B) =
  0$ otherwise, we have \[ \lambda_s (f_k - f_s)
  = \sum_{B, E \subseteq B} \sum_{x, x \notin E} \chi(x, B) 
  = \sum_{x, x \notin E} \sum_{B, E \subseteq B}\chi(x, B) 
  = (f_n - f_s) \lambda , \] 
  since if $x \notin E \subseteq B$, then $x \in B$ if and only if
  $\cl(E \cup \{ x \}) \subseteq B$, which is a flat of rank $s + 1 =
  t$.  Hence $\lambda_s = \lambda \frac{f_n - f_s} {f_k - f_s}$,
  independent of~$E$, as desired.
\end{proof}


\section{Incidence geometry}\label{sec:geometry}

The projective space and the affine space will provide natural classes
of perfect matroid designs.  We present here briefly the synthetic
approach to these geometries, cf.~\cite{Be95,BR98}.  For simplicity
all sets are assumed to be finite.

\begin{defn}
  An \emph{incidence space} is a triple $\bG = (\cP, \cL, I)$,
  where~$\cP$ and~$\cL$ are sets, the elements of which are referred
  to as \emph{points} and \emph{lines}, respectively, and $I \subseteq
  \cP \times \cL$ is a relation, called \emph{incidence}, satisfying
  the following axioms:
  \begin{enumerate}[(i)]
  \item for each pair of distinct points~$P$ and~$Q$ there is a unique
    line that is incident with~$P$ and~$Q$, denoted by $PQ$;
  \item each line is incident with at least two points.
  \end{enumerate}
\end{defn}

For a point~$P$ and a line~$\ell$ with $(P, \ell) \in I$ we say that
$P$ is incident with~$\ell$, $P$ is on~$\ell$, $\ell$ is incident
with~$P$, or $\ell$ goes through~$P$.  Two lines $\ell_1, \ell_2 \in
\cL$ \emph{meet} if there is a point~$P$ on both~$\ell_1$ and~$\ell_2$.

For a line~$\ell$, denote by $(\ell)$ the set of points incident
with~$\ell$.  A set~$\cU$ of points is called a \emph{linear set} if
$(PQ) \subseteq \cU$ for any distinct points $P, Q$ of~$\cU$.  Such a
set defines an incidence space $\bU = (\cU, \cL', I')$ in a natural
sense, called a \emph{subspace} of~$\bG$.  If the context is clear, we
may identify a linear set and its corresponding subspace.  The linear
sets form a \emph{closure system}, i.e., they are closed under
arbitrary intersection.  Given a subset~$\cX$ of~$\cP$ its \emph{span}
$\cl(\cX)$ is defined as the smallest linear set containing~$\cX$.

\begin{defn}
  A \emph{projective space} is an incidence space $\bP = (\cP, \cL,
  I)$ with the following properties:
  \begin{enumerate}[(i)]
  \item if $A, B, C, D$ are four points such that the lines $AB$ and
    $CD$ meet, then also the lines and $AC$ and $BD$ meet
    (\emph{Veblen-Young axiom});
  \item each line is incident with at least three points.
  \end{enumerate}
\end{defn}

It is easy to see that any subspace of a projective space is again a
projective space.  The family of all subspaces of a projective space
with its natural incidence relation given by containedness is also
called the \emph{projective geometry}.

In any projective space the exchange property of the span holds, i.e.,
for any linear set~$\cU$ and points $P, Q \notin \cU$ one has
\[ Q \in \cl( \cU \cup \{ P \} ) \ \text{ implies } \ P \in \cl( \cU
\cup \{ Q \} ) . \]
Defining a set of points~$\cB$ to be \emph{independent} if $P \notin
\cl( \cB \setminus \{ P \} )$ for all $P \in \cB$, then it follows
from Proposition~\ref{prop:closure} that $M_{\bP} \coloneqq (\cP,
\cI)$ is a matroid, where~$\cI$ is the family of all independent sets.

An independent set~$\cB$ that spans~$\cP$ is called a \emph{basis}.
The cardinality of a basis~$\cB$ is independent of its choice, and one
defines $\dim \bP = | \cB | - 1$ to be the \emph{dimension} of~$\bP$.
Accordingly, each linear subset~$\cU$ is associated with a dimension,
and one easily sees that a single point has dimension zero and a line
dimension one; a two-dimensional subspace is called a \emph{plane},
and a subspace of dimension $\dim \bP - 1$ is a \emph{hyperplane}.  We
note that $\dim \bP = \rho( M_{\bP} ) - 1$, where~$\rho$ is the
matroid rank, and the $k$-flats are precisely the $(k-1)$-dimensional
subspaces. \pagebreak

\begin{defn}
  An \emph{affine space} is an incidence space $\bA = (\cP, \cL, I)$
  such that the following axioms are satisfied:
  \begin{enumerate}[(i)]
  \item there exists a \emph{parallelism} on~$\bA$, i.e., an
    equivalence relation~$\parallel$ on~$\cL$ such that for each
    point~$P$ and each line~$g$ there exists a unique line~$h$
    through~$P$ such that $h \parallel g$, denoted by $P \parallel g$;
  \item if $A, B, C$ are noncollinear points and if $A', B'$ are
    points with $AB \parallel A'B'$, then the lines $A' \parallel AC$
    and $B' \parallel BC$ intersect in a point~$C'$ (\emph{triangle
      axiom}).
  \end{enumerate}
\end{defn}

Again, the span in affine space satisfies the exchange property.  Thus
the family of independent sets of an affine space~$\bA$ form a
matroid~$M_{\bA}$, and all bases have the same cardinality; the
\emph{dimension} of~$\bA$ is $\dim \bA = | \cB | - 1 = \rho( M_{\bA} )
- 1$, where~$\cB$ is a basis of~$\bA$.

There is a close relation between projective and affine spaces; in
fact, every projective space induces an affine space and vice versa.

\begin{prop}
  Let~$\bP$ be a projective space and~$\cH$ a hyperplane in~$\bP$.
  Then $\bA = (\cP', \cL', I')$, where $\cP' \coloneqq \cP \setminus
  \cH$, $\cL' \coloneqq \{ \ell \in \cL \mid \ell \not\subseteq \cH
  \}$ and $I' \coloneqq I \cap (\cP' \times \cL')$, defines an affine
  space.
\end{prop}

\begin{prop}
  Let~$\bA = (\cP, \cL, I)$ be an affine space.  Let $\cH \coloneqq \{ [
  \ell ] \mid \ell \in \cL \}$, where $[ \ell ]$ denotes the parallel class
  of the line~$\ell$.  For $\ell \in \cL$ let $\ell' \coloneqq \ell \cup \{
  [ \ell ] \}$, and for each plane~$\alpha$ in~$\bA$ let $\ell_{\alpha}
  \coloneqq \{ [ \ell ] \mid \ell \text{ line in } \alpha \}$.  Define
  \[ \cP' \coloneqq \cP \cup \cH \quad \text{ and } \quad
  \cL' \coloneqq \{ \ell' \mid \ell \in \cL \} \cup \{ \ell_{\alpha} \mid
  \alpha \text{ plane in } \bA \}. \] 
  Then $\bP = (\cP', \cL', I')$ (with the obvious incidence) is a
  projective space in which~$\cH$ is a hyperplane.
\end{prop}

The cardinalities of subspaces in finite projective or affine spaces
turn out to follow a regular pattern.  A first result is that every
line in a projective space~$\bP$ is incident with the same number of
points, and one calls $q \coloneqq | (\ell) | - 1 \ge 2$, where~$\ell$
is a line, the \emph{order} of~$\bP$.

\begin{thm}\label{thm:pmd-pro}
  Any finite projective space of dimension~$d$ and order~$q$ has
  exactly $q^d + \dots + 1 = \smash{\frac{q^{d+1} - 1} {q - 1}}$
  points.  Therefore, each $t$-dimensional subspace has $q^t + \dots +
  1$ points.
\end{thm}

Using the projective closure, one sees that all lines in a finite
affine space~$\bA$ have also the same number of points, and $q
\coloneqq | (\ell) | \ge 2$, where~$\ell$ is a line, is called the
order of~$\bA$.

\begin{thm}\label{thm:pmd-aff}
  Any finite affine space of dimension~$d$ and order~$q$ has
  exactly~$q^d$ points, and each $t$-dimensional subspace has~$q^t$
  points.
\end{thm}

As a consequence of Theorem~\ref{thm:pmd-pro} and Theorem%
~\ref{thm:pmd-aff} one sees that both in projective space and in
affine space the matroid of independent sets is a perfect matroid
design, and we can therefore define \emph{projective} and {\em
  affine}~designs using Definition~\ref{defn:design-matroid}.


\subsection*{Coordinization}

Linear algebra provides a wealth of examples for projective spaces.
Given a vector space~$V$, let $\bP(V) \coloneqq (\cP, \cL, I)$ be the
incidence structure, where~$\cP$ is the set of one-dimensional
subspaces of~$V$, $\cL$ is the set of two-dimensional subspaces and
the incidence~$I$ is given by containedness.

\begin{prop}
  If~$V$ is an $n$-dimensional vector space over the finite
  field~$\F_q$, then~$\bP(V)$ is a projective space of dimension $d
  \coloneqq n - 1$ and order~$q$, denoted by $\PG(d, q)$.  Its
  subspaces~$\bU$ of dimension $t \coloneqq k - 1$ correspond to
  vector subspaces $W \le V$ of dimension~$k$.
\end{prop}

Conversely, the first structure theorem for projective space implies
the remarkable fact that any projective space~$\bP$ of dimension $\ge
3$ is of the form $\bP(V)$ for a vector space~$V$ over a (skew) field,
cf.~\cite[Cor.~3.4.3]{BR98}.

For a vector space~$V$ one can also define the affine space $\bA(V)
\coloneqq (\cP, \cL, I)$, where the points~$\cP$ are (identified with)
the vectors in~$V$, the lines~$\cL$ are the cosets of one-dimensional
subspaces of~$V$ and the incidence~$I$ is given by containedness.

\begin{prop}
  For a $d$-dimensional vector space~$V$ over~$\F_q$ the incidence
  structure $\bA(V)$ is an affine space of dimension~$d$ and
  order~$q$, denoted by $\AG(d, q)$.  Its subspaces~$\bU$ of
  dimension~$t$ are given by the cosets of $t$-dimensional subspaces
  of~$V$.
\end{prop}


\section{Designs in finite geometries}\label{sec:designs}

We are now prepared to state the definition of a projective design and
of an affine design.  Recall that in projective or affine space the
(matroid) rank equals the geometric dimension plus one.

\begin{defn}\label{defn:design-geometry}
  Let $\bG = (\cP, \cL, I)$ be a projective space or an affine space
  of rank~$n$ (dimension~$n-1$).  A $t$-$(n, k, \lambda)$
  \emph{design} in~$\bG$ is a collection~$\cB$ of $k$-flats (i.e.,
  $(k-1)$-dimensional subspaces) of~$\bG$, called \emph{blocks}, such
  that every $t$-flat (i.e., $(t - 1)$-dimensional subspace) of~$\bG$
  is contained in exactly~$\lambda$ blocks.

  If $\bG = \bP$ is a projective space we also refer to a $t$-$(n, k,
  \lambda)$ \emph{projective design}, and in case $\bG = \bA$ is an
  affine space to an $t$-$(n, k, \lambda)$ \emph{affine design}.  If
  $\lambda = 1$ we speak of a (projective or affine) \emph{Steiner
    system} $S(t, k, n)$.
\end{defn}

Note that this definition is consistent with
Definition~\ref{defn:design-matroid} of designs in matroids.  If the
projective space is induced by a vector space we can translate
Definition~\ref{defn:design-geometry} into a more common form as
follows.  Let~$V$ be an $n$-dimensional vector space over~$\F_q$, so
that $\bP(V) = \PG(n - 1, q)$.  A $t$-$(n, k, \lambda)$ \emph{subspace
  design} of order~$q$ is a collection~$\cB$ of $k$-dimensional
subspaces of~$V$ such that every $t$-dimensional subspace of~$V$ is
contained in exactly~$\lambda$ members of~$\cB$.


\subsection*{Affine designs from projective designs}

A \emph{translation} of a vector space~$V$ is a map $\alpha: V \to V$
of the form $x \mapsto v + x$ for some $v \in V$.  The set of all
translations defines a group~$T$ under composition, which is
isomorphic to~$V$ as an abelian group.

A nonempty subset $W \subseteq V$ is an affine space if and only if $W
= \alpha U$ for some subspace~$U$ and a translation~$\alpha$, i.e., if
$W = v + U$ for some $v \in V$.  In this case, $\dim U$ is the
dimension of the affine space~$W$.  Note that an affine space is a
vector subspace if and only if it contains~$0$; in fact, if $0 \in
\alpha U$, then $\alpha U = U$.

\begin{thm}\label{thm:pro-aff}
  Suppose that~$\cB$ is a $t$-$(n, k, \lambda)$ subspace design
  in~$V$, then $T \cB \coloneqq \{ \alpha U \mid U \in \cB, \alpha \in
  T \}$ is an $(t+1)$-$(n+1, k+1, \lambda)$ affine design in $\bA(V)$.
  Conversely, if~$\cD$ is an $(t+1)$-$(n+1, k+1, \lambda)$ affine
  design, then $\cD_0 \coloneqq \{ W \in \cD \mid 0 \in W \}$ is a
  $t$-$(n, k, \lambda)$ subspace design.
\end{thm}

\begin{proof}
  We prove the latter statement first.  All members $W \in \cD_0$ are
  $k$-dimensional subspaces.  If $X \le V$ is any $t$-dimensional
  subspace, then~$X$ is also a $t$-dimensional affine space.  As $0
  \in X$, clearly $\{ W \in \cD \mid X \subseteq W \} = \{ W \in
  \cD_0 \mid X \subseteq W \}$ and by assumption the cardinality of
  this set is~$\lambda$.  This shows that~$\cD_0$ is a $t$-$(n, k,
  \lambda)$ subspace design.
  
  Now consider the other statement.  All members $W \in T \cB$ are
  certainly $k$-dimensional affine spaces.  Let $Y = \alpha X$ be any
  $t$-dimensional affine space, where $X \le V$ is a $t$-dimensional
  subspace and $\alpha \in T$.  We claim that there is a bijection $\{
  U \in \cB \mid X \subseteq U \} \to \{ W \in T \cB \mid \alpha X
  \subseteq W \}$ given by $U \mapsto \alpha U$.  Clearly, $X
  \subseteq U$ implies $\alpha X \subseteq \alpha U$, thus the map is
  defined and obviously injective.  On the other hand, if $W = \beta U
  \in T \cB$ satisfies $\alpha X \subseteq \beta U$, then $0 \in X
  \subseteq \alpha^{-1} \beta U$ and hence $\alpha^{-1} \beta U = U$;
  thus $W = \beta U = \alpha U$ and $X \subseteq U$, which shows
  surjectivity.  Therefore, the sets have the same cardinality%
  ~$\lambda$, which shows that~$T \cB$ is an affine design.
\end{proof}

\begin{prop}\label{prop:q-cl}
  For any $2$-$(n, k, \lambda)$ subspace design there is a
  $2$-$([n]_q, [k]_q, \lambda)$ classical design, where $[d]_q \coloneqq
  \smash{\frac{q^d - 1} {q - 1}}$.  For any $2$-$(n, k, \lambda)$
  affine design there is a $2$-$(q^{n-1}, q^{k-1}, \lambda)$ classical
  design, and for any $3$-$(n, k, \lambda)$ affine design of order $q
  = 2$ there is a $3$-$(2^{n-1}, 2^{k-1}, \lambda)$ classical design.
\end{prop}

\begin{proof}
  For the subspace design, let~$V$ be an $n$-dimensional vector space
  and consider the points~$\cP$ of its projective space $\bP(V)$,
  i.e., the set of all its one-dimensional subspaces; thus $|\cP| =
  [n]_q$.  Each block~$U$ of a $2$-$(n, k, \lambda)$ subspace design
  corresponds to a subset $\cU \coloneqq \{ P \in \cP \mid P \subseteq
  U \}$ of cardinality~$[k]_q$, and a two-element set $\{ P, Q \}$
  of~$\cP$ is contained in such a set~$\cU$ if and only if the
  two-dimensional subspace spanned by~$P$ and~$Q$ is contained in a
  block~$U$.  Thus any $2$-set of~$\cP$ is contained in
  exactly~$\lambda$ blocks.

  Now let~$V$ be a vector space of dimension $n - 1$ and consider a
  design in $\bA(V)$ consisting of $(k - 1)$-dimensional affine
  spaces.  Then~$V$ has $q^{n-1}$ elements, each block has $q^{k-1}$
  elements, and a block contains a two-element set~$X$ if and only if
  it contains the affine space (a line) generated by~$X$.  This shows
  the claim for $2$-$(n, k, \lambda)$ affine designs.

  Finally, if the ground field of the vector space~$V$ is~$\F_2$, then
  any three points are affine independent, i.e., they span a
  two-dimensional affine space (a plane).  As three points are in a
  block if and only if their generated plane is contained in the
  block, the last claim follows.
\end{proof}

By combining Theorem~\ref{thm:pro-aff} and Proposition~\ref{prop:q-cl}
we immediately get:

\begin{cor}[{cf.~\cite[Th.~4]{EV11}}]
  For any $2$-$(n, k, \lambda)$ subspace design of order~$2$ there is
  a $3$-$(2^n, 2^k, \lambda)$ classical design.  In particular, for
  any $S(2, 3, n)$ subspace Steiner system one obtains a classical
  $S(3, 8, 2^n)$ Steiner system.
\end{cor}


\subsection*{Existence of affine Steiner systems}

Recall that a \emph{spread} in a projective geometry~$\bP$ of rank~$n$
(dimension~$n-1$) is a projective Steiner system $S(1, k, n)$, i.e., a
partition of the point set of~$\bP$ into subspaces of rank~$k$
(dimension~$k-1$).  It is well-known that for any order~$q$ a spread
exists if and only if $k \mid n$ (cf.~\cite[p29]{De68}).  Therefore,
Theorem~\ref{thm:pro-aff} readily implies the following result.

\begin{prop}\label{prop:steiner}
  For any $k, \ell$ there is an affine Steiner system $S(2, k+1, k
  \ell + 1)$.
\end{prop}

This shows, of course, for any order~$q$ and for all~$\ell$ the
existence of Steiner triple systems $S(2, 3, 2 \ell + 1)$ in affine
geometry, and in particular the ``affine $q$-analog'' of the Fano
plane $S(2, 3, 7)$.

\begin{exa}
  Consider an affine Steiner system $S(2, 3, 7)$ for $q = 2$.  This is
  a family~$\cB$ of planes in $\bA( \F_2^6 )$ such that each line is
  contained~in exactly one plane in~$\cB$.  As there are $2016$ lines
  in $\bA( \F_2^6 )$ the size of~$\cB$ is $\frac 1 6 \cdot 2016 = 16
  \cdot 21 = 336$.
\end{exa}
    
Note that the affine Steiner system $S(2, 3, 7)$ constructed from a
spread $S(1, 2, 6)$ in $\PG(5, 2)$ cannot be extended to a subspace
code in the Grassmannian $\cG_2(7, 3)$ of minimum distance~$4$ (i.e.,
to a partial projective Steiner system).  Indeed, this affine Steiner
system $S(2, 3, 7)$ features only~$21$ parallel classes, while the
projective closures of parallel planes will intersect in a common line
at infinity.

In this context we mention that there is a different affine Steiner
system $S(2, 3, 7)$, which is invariant under the Singer cycle of
size~$63$ and which has~$273$ parallel classes.  It has been used in
\cite[Lem.~6]{EV11} to construct a covering code in $\cG_2(7, 3)$ of
size~$399$, covering the $2$-dimensional subspaces.

\begin{que}
  Does there exist an affine Steiner system $S(2, 3, 7)$ for $q = 2$,
  which is \emph{skew}, i.e., with no pair of parallel planes?  If
  yes, then a new subspace code in $\cG_2(7, 3)$ of distance~$4$ and
  size~$336$ is found, improving on the largest size~$333$ of such a
  code reported so far~\cite{HKKW16}.  If no, then the long-standing
  open problem on the existence of the projective $q$-analog of the
  Fano plane $S(2, 3, 7)$ (of size~$381$) would be settled, as such a
  structure cannot exist in this case.
\end{que}


\section{Random network coding}\label{sec:coding}

In this final section we discuss possible applications of affine
designs in a random network coding scenario.  It is straightforward to
adapt the concept of random network coding~\cite{H+06} for affine
spaces, by stipulating that the inner nodes in a noncoherent network
forward a random \emph{affine} combination of the incoming symbols,
instead of a linear combination.  That is, if $v_1, \dots, v_s \in V$
are the received vectors on the incoming edges of a node, the output
along any outgoing edge is an affine combination $\sum_{i=1}^s
\lambda_i v_i$, where the coefficients~$\lambda_i$ are randomly chosen
such that $\sum_{i=1}^s \lambda_i = 1$, i.e., $\lambda_s = 1 -
\sum_{i=1}^{s-1} \lambda_i$.  This is detailed and analyzed
in~\cite{GG11}, where it is shown that affine network coding saves
about one symbol when compared to the standard linear network coding
approach.

Regarding error-correction it is customary in coding theory to use a
metric space in order to model the distance between the sent and the
possibly altered received codeword.  We consider thus metrics in
geometries and related concepts below.

Let~$\bG$ be a projective or an affine geometry of rank~$n$, i.e., of
dimension $n - 1$.  Then, as noted after Proposition~\ref{prop:metric},
a metric on its set of flats is given by 
\[ d(E, F) \coloneqq 2 r( E \vee F) - r(E) - r(F) , \] 
where $E, F$ are flats and $r$ denotes the (matroid) rank.  A
\emph{small-intersection code} or a \emph{partial $S(t, k, n)$ Steiner
  system} in~$\bG$ is a collection~$\cC$ of $k$-flats such that $r(E
\cap F) < t$ for all distinct flats $E, F \in \cC$.  Clearly, if an
$S(t, k, n)$ Steiner system exists it is a small-intersection code of
maximal cardinality.

In projective geometry the \emph{modular equality} $r(E \vee F) + r(E
\wedge F) = r(E) + r(F)$ holds, and thus we can write for the
matroid metric \[ d(E, F) = 
r(E \vee F) + r(E \wedge F) = r(E) + r(F) - 2 r(E \wedge F) , \]
which for $\bG = \bP(V)$ equals the subspace distance of the
corresponding subspaces in~$V$.  In this case, a small-intersection
code is just a code~$\cC$ in the Grassmannian $\cG(n, k)$ of minimum
distance $d(\cC) \ge 2 (k - t + 1)$, and is usually simply referred to
as a \emph{subspace code}.

On the other hand, in affine geometry the modular equality does not
hold.  In fact, the formula
\[ d_{\wedge} (E, F) \coloneqq r(E) + r(F) - 2 r(E \wedge F) \]
does not define a metric, as the triangle inequality is not satisfied
in general.  For example, if $E, F$ are parallel planes and~$T$ is a
further plane that intersects both~$E$ and~$F$ in a line, then
$d_{\wedge} (E, F) = 6$, but $d_{\wedge}(E, T) + d_{\wedge}(T, F) = 2
+ 2 = 4$.


\subsection*{Recovery of deletions}

Notwithstanding the above remark, small-intersection codes in affine
geometry (and thus affine Steiner systems) are interesting in a
network coding scenario for the correction of \emph{deletions}, i.e.,
in a situation when during transmission no erroneous vectors are being
injected into the network, but some information may be lost due to
lack of connectivity.

The abstract framework for adverserial error correction
of~\cite[Sec.~III]{SK09} can be applied in the present situation.
Let~$L$ be the lattice of flats in an affine (or a projective)
geometry.  Define the \emph{deletion discrepancy} $\Delta: L \times
L \to \N \cup \{ \infty \}$ by
\[ \Delta(E, F) \coloneqq \begin{cases} r(E) - r(F) & \text{if } F
  \subseteq E ,\\ \infty & \text{otherwise} . \end{cases} \]
The following result is immediate from~\cite[Prop.~1]{SK09}.

\begin{prop}
  A code $\cC \subseteq L$ is $e$-deletions-correcting if
  $e \le \min_{E \ne E' \in \cC} \tau(E, E')$, where $\tau(E, E')
  = \min_{F \in L} \max \{ \Delta(E, F), \Delta(E', F) \} - 1$.
\end{prop}

In particular, if $\cC$ is a small-intersection-code with parameters
$(t, k, n)$, then $\tau(E, E') = k - r(E \cap E') - 1 \ge k - t$ for
any distinct $E, E' \in \cC$.  Hence such a code is able to correct up
to~$e = k - t$ deletions, i.e., when up to $k - t$ independent vectors
less than submitted are obtained by the receiver.


\subsection*{A polynomial-based code construction}

In affine geometry, a construction of good codes with respect to the
metric~$d$ has been proposed based on a lifting of maximum
rank-distance codes~\cite{GG11}.  Here we present some good
small-intersection codes, which follow the construction of subspace
codes based on linearized polynomials~\cite{KK08} and allow for a
slightly wider range of parameters.

Recall that a \emph{linearized polynomial} is a polynomial $f \in
\F_{q^m}[X]$ of the form $f = f_0 X + f_1 X^q + \dots + f_r X^{q^r}$
for some $f_i \in \F_{q^m}$; these are exactly the polynomials~$f$ (of
degree $< q^m$), for which the associated map $\F_{q^m} \to \F_{q^m}$,
$x \mapsto f(x)$, is $\F_q$-linear.  An \emph{affine polynomial} is a
polynomial $g \in \F_{q^m}[X]$ of the form $g = a + f$ with $a \in
\F_{q^m}$ and a linearized polynomial $f \in \F_{q^m}[X]$; these are
the polynomials~$g$ such that their corresponding function is
$\F_q$-affine, i.e., $g( \lambda x + \mu y ) = \lambda g(x) + \mu
g(y)$ for $x, y \in \F_{q^m}$ and $\lambda, \mu \in \F_q$ with
$\lambda + \mu = 1$.

For $t \ge 1$ let $\cL_t \coloneqq \{ a + \sum_{i=0}^{t-2} f_i X^{q^i}
\mid a, f_i \in \F_{q^m} \}$ be the set of affine polynomials up to
degree $q^{t-2}$.  Let $t - 1 \le \ell \le m$ and let~$U$ be an
$\F_q$-affine subspace of~$\F_{q^m}$ of dimension~$\ell$.  Define
\[ \cC \coloneqq \{ \Gamma( g|_U ) \mid g \in \cL_t \}, \]
where $\Gamma( g|_U ) \subseteq U \times \F_{q^m}$ denotes the graph
of $g|_U \colon U \to \F_{q^m}$, $x \mapsto g(x)$.

\begin{prop}\label{prop:code}
  The code~$\cC$ has $q^{mt}$ elements and is a partial $S(t, k, n)$
  Steiner system in $\bA(V)$, where $V \coloneqq U \times \F_{q^m}$
  and $n \coloneqq r(V) = \ell + m + 1$, $k \coloneqq r(X) = \ell + 1$
  and $r(X \cap Y) < t$ for distinct $X, Y \in \cC$.
\end{prop}

\begin{proof}
  For different polynomials $g, h \in \cL_t$, due to the maximum
  degree the functions $g|_U$ and $h|_U$ coincide on at most $q^{t-2}
  < q^{\ell}$ elements, which shows that $X = \Gamma(g|_U)$ and $Y =
  \Gamma(h|_U)$ are distinct and $r(X \cap Y) = \dim(X \cap Y) + 1 \le
  t - 1$.  Then the rest is clear.
\end{proof}

\begin{exa}
  Let $q = 2$.  Considering, say, $m = \ell = 3$ and $t = 3$ we get
  from Proposition~\ref{prop:code} a small-intersection code in
  $\AG(6, q)$, which is a partial $S(3, 4, 7)$ Steiner system with
  $q^9 = 512$ elements.  For $t = 2$ we obtain accordingly a partial
  $S(2, 4, 7)$ Steiner system with $q^6 = 64$ elements, and by
  Proposition~\ref{prop:steiner} there is even a full $S(2, 4, 7)$
  Steiner system, which has~$72$ elements.

  Comparing the codes with their siblings in projective geometry, one
  sees from~\cite{HKKW16} that the best subspace codes with these
  parameters have fewer elements, namely $A_2(n, d; k) = A_2(7, 4; 4)
  = A_2(7, 4; 3) \le 381$ and $A_2(7, 6; 4) = A_2(7, 6; 3) = 17$.
\end{exa}


\section*{Acknowledgments}

The author would like to thank Stefan E.\ Schmidt and Anna-Lena
Horlemann-Trautmann for helpful discussions, and gratefully
acknowledges continuous support from COST Action IC1104 ``Random
Network Coding and Designs over $\on{GF}(q)$''.


\end{document}